\documentclass[final,10pt,3p,times,twocolumn]{elsarticle}

\usepackage{amssymb}

\usepackage{amsmath}
\usepackage{amssymb}
\usepackage{amsfonts}
\usepackage{amsthm}
\usepackage{graphicx}
\usepackage{soul}
\usepackage[colorlinks=true, allcolors=blue]{hyperref}

\usepackage{tikz}
\usepackage{caption, subcaption}
\usepackage{cleveref}

\usepackage{graphics}
\usepackage{array}

\newtheorem{theorem}{Theorem}[section]
\newtheorem{corollary}[theorem]{Corollary}
\newtheorem{lemma}[theorem]{Lemma}

\newtheorem{remark}[theorem]{Remark}

\theoremstyle{definition}

\newcommand{\leo}[1]{\textcolor{black!50!black}{#1}}
\definecolor{lightgreen}{RGB}{180,255,150}

\newcommand{\vince}[1]{\textcolor{black!50!black}{#1}}
\newcommand{\rev}[1]{\textcolor{black!50!black}{#1}}

\definecolor{lightblue}{RGB}{180,180,255}

\journal{Information Processing Letters}

\begin{document}

\begin{frontmatter}



\title{Polynomial invariants for cactuses}


\author[tu]{Leo van Iersel}
\ead{l.j.j.vaniersel@tudelft.nl}

\author[uea]{Vincent Moulton}
\ead{v.moulton@uea.ac.uk}

\author[tu]{Yukihiro Murakami}
\ead{y.murakami@tudelft.nl}

\address[uea]{School of Computing Sciences, University of East Anglia, NR4 7TJ, Norwich, United Kingdom}

\address[tu]{Delft Institute of Applied Mathematics, Delft University of Technology, Mekelweg 5, 2628 CD, Delft, the Netherlands}




\begin{abstract}
Graph invariants are a useful tool in graph theory.
Not only do they 
encode useful information about the graphs to which they are associated, but 
complete invariants can be used to distinguish 
between non-isomorphic graphs.
Polynomial invariants for graphs such as the well-known Tutte polynomial
have been studied for several years, and recently there has been
interest to also define such invariants for phylogenetic networks,
a special type of graph that arises in the
area of evolutionary biology. Recently Liu gave a complete invariant for (phylogenetic) trees. 
However, the polynomial invariants defined thus far for
phylogenetic networks that are not trees require vertex labels and either contain a large number of variables, or they have exponentially many terms in the number of reticulations. 
This can make it difficult to compute these polynomials and to use them to analyse unlabelled networks.
In this paper, we shall show how to circumvent some of these 
difficulties for rooted cactuses and  cactuses.
As well as being important in other areas
such as operations research, rooted cactuses
contain some common classes of phylogenetic networks such
phylogenetic trees and level-1 networks.
More specifically, we define a polynomial
$F$ that is a complete invariant for the class of rooted cactuses
without vertices of indegree 1 and outdegree 1 that has 5 variables,
and a polynomial $Q$ that is a complete invariant
for the class of rooted cactuses
that has 6 variables \vince{whose degree can be bounded linearly in terms 
of the size of the rooted cactus}. We also explain how to extend the~$Q$ polynomial
to define a complete invariant for leaf-labelled rooted cactuses
as well as (unrooted) cactuses.
\end{abstract}


\end{frontmatter}




\section{Introduction}

Given a class $\mathcal C$ of graphs, and a polynomial $P(C)$ assigned to each
element $C$ in this class, we call $P$ an {\em invariant} of $\mathcal C$
if $P(C)=P(C')$ when $C$ is isomorphic to $C'$ for all $C,C' \in \mathcal C$; if 
$P(C)=P(C')$ also implies that $C$ is isomorphic to $C'$ we call $P$
a {\em complete invariant} for the class (see e.g. \citep{liu2019generalized}).
Polynomial invariants have been defined for various classes of graphs
\vince{(see e.g. \cite{bollobas1998modern}),}
including the extensively studied {\em Tutte polynomial} 
(see e.g. \citep{awan2020tutte} and the references 
therein).  These usually encode useful information about the 
graph (e.g. number of edges, spanning forests), but they are not always 
complete invariants, and finding complete invariants
for graphs remains an important area of research in graph theory 
and computer science \citep{liu2019generalized}.

Recently there has been interest in defining polynomial invariants 
for graphs that arise in the field of phylogenetics~\citep{liu2021tree,pons2022polynomial}.
Such graphs are called \emph{phylogenetic networks}, and 
they often come equipped with a leaf-labelling of the 
vertices corresponding to some collection 
of species (see e.g. \cite[Chapter 10]{steel2016phylogeny} for a recent overview).  
Phylogenetic networks are commonly 
used to elucidate the evolutionary history for a collection of species
that has undergone non-treelike evolution (such as e.g. bacteria or plants)~\citep{bapteste2013networks}, 
and an important problem asks to find ways to distinguish between
distinct networks to compare evolutionary 
histories~\vince{(see e.g. \citep{cardona2008metrics})}.
\vince{Hence} finding complete invariants for special classes
of phylogenetic networks could be useful as, for
example, they would yield metrics on networks in question
(see e.g. \citep{liu2022analyzing} where this 
approach has been recently used to 
study how the influenza virus evolves).

In this paper, we focus on the problem of finding 
complete polynomial invariants for the classes of rooted cactuses
and cactuses
(with or without leaf-labellings), 
two special classes of phylogenetic networks~\citep{hayamizu2020recognizing,huber2021space} 
which also
arise in other areas such as \vince{operations research ~\citep{kang20142}} and genome comparisons~\citep{paten2011cactus}. 
More specifically, recall
that a {\em cactus} (also known as a \emph{Husimi} tree) is a connected undirected graph in which
any two cycles are edge disjoint~\citep{husimi1950note};
a {\em rooted cactus} is a directed acyclic graph
with a single source or root whose
underlying undirected graph is a cactus, and 
such that there is a directed path from the root to any vertex in 
the graph~\citep{huber2021space}
(see \Cref{fig:cactus} for an example of a cactus and a rooted cactus).
Observe that rooted cactuses are related to, but different from \emph{directed cactuses},
which are strongly connected directed graphs where each edge is contained in exactly one directed cycle~\citep{balaji2020resistance}.
\vince{Also note that in general a rooted graph may also mean any type of graph 
that contains a distinguished vertex, which is slightly different from our meaning of rooted.}
If the underlying graph of a rooted cactus is an 
undirected tree (or, for brevity, a tree), we
call it a {\em rooted tree} (also known as an {\em arborescence}). 
As well as trees, there are various other subclasses 
of unrooted and rooted cactuses such as
(unlabelled) level-1 networks and galled trees
\vince{(see e.g. \cite[p.247]{steel2016phylogeny})}.

In previous related work, for \vince{trees}, one 
of the first complete invariants for rooted 
trees was introduced in 
\citep{gordon1989greedoid}. In fact, this polynomial is defined as 
the restriction of a certain greedoid polynomial (the Tutte polynomial of a greedoid)
which arises from the fact that any rooted directed graph gives rise to
a certain greedoid structure \citep{gordon1989greedoid}.
In \citep{tedford2009tutte} a
modification was made to this polynomial to also give a complete invariant for rooted undirected unicyclic graphs
(note that rooted undirected unicyclic graphs are similar but 
different from rooted cactuses, as the former is undirected and
contains at most one cycle, while the 
latter is directed and can contain more than one cycle in the underlying graph).
Interestingly, the problem of defining a complete invariant
for (undirected) trees remained open until a solution was recently proposed 
in \citep{liu2021tree}, in which Liu defined
a new complete invariant $B$ for rooted trees, which was extended to the class of
\vince{(undirected) trees}.

More recently, polynomial 
invariants have also been introduced for classes of phylogenetic networks,
building on Liu's approach. In \citep{liu2021tree},
Liu showed how to extend the polynomial invariant $B$ to leaf-labelled trees, and
more recently a
polynomial invariant was introduced for 
rooted binary internally multi-labelled phylogenetic networks 
(where vertices with indegree at least~2, or \emph{reticulations},
are distinctly labelled) in \citep{pons2022polynomial}.
This was shown to be a complete invariant 
for a certain subclass of such partly-labelled networks.
In~\citep{janssen2021comparing},
a polynomial invariant is defined for rooted unlabelled networks based on their spanning trees,
and it is briefly mentioned that the polynomial is a complete invariant 
restricted to the class of so-called rooted, 
leaf-labelled\footnote{In \citep{janssen2021comparing} this result is stated to hold for 
unlabelled networks, but in a
personal correspondence with the authors we 
were informed that it should be stated to hold for labelled networks.}
{\em tree-child} networks.

The polynomial invariants defined thus far for
phylogenetic networks (that are not trees) require vertex labels and either contain a large number of variables, or they have exponentially many terms in the number of reticulations. 
This can make it difficult to compute these polynomials and to use them to analyse unlabelled networks.
In this paper, we shall show how to circumvent some of these 
difficulties for rooted cactuses and  cactuses. More specifically, we
begin in \Cref{sec:prelim} by introducing a complete invariant for a
certain class of vertex-labelled trees (which we call
\emph{special pairs}), which generalizes the tree polynomial presented in \citep{liu2021tree}. 
We then use this new invariant to introduce two new polynomial invariants for 
rooted cactuses.

Our first polynomial invariant $F$ given in 
\Cref{sec:unfolding} is based on \emph{unfolding} a network, a technique
used in  \citep{pons2022polynomial} and inspired by~\citep{huber2016folding}.
When unfolding a rooted cactus, copies of the subnetwork rooted at every 
reticulation are created.
Labels are subsequently added to the unfolded network, which is
a directed tree, to create a special pair. In Theorem~\ref{thm:cactus-invariant}
we give a one-to-one 
correspondence between a rooted cactus and its labelled unfolded network.
This immediately gives a polynomial invariant $F$
for rooted cactuses using the special pair definition from before. 
The polynomial $F$ has 5 variables when 
leaves are unlabelled, and at most~$n+4$ variables when its~$n$ leaves are labelled.
In addition, as we shall show, $F$ is a 
complete invariant for the class of rooted cactuses 
without elementary vertices, that is vertices with 
indegree and outdegree one. Note that in phylogenetics this is not
a strong assumption since elementary vertices are 
commonly excluded from phylogenetic networks
as they do not correspond to evolutionary events such as speciation or hybridization.

Our second polynomial invariant $Q$ for rooted cactuses which is 
given in \Cref{sec:expanding},
is based on \emph{expanding} a rooted cactus, a process in which we 
encode every reticulation 
with two appended leaves and an added edge.
Each iteration of the expansion removes one reticulation 
vertex from the network; 
labelling the vertices creates a special pair, that 
uniquely encodes the original rooted cactus.
As we prove in Theorem~\ref{thm:cactus-invariant2},
this gives a complete invariant for the class 
of rooted cactuses (where we now allow for elementary vertices), which we call~$Q$.
The $Q$ polynomial has 6 variables when leaves are unlabelled, and at most~$n+5$ variables when its~$n$ leaves are labelled.
We also show that the degree of the~$Q$ polynomial is 
linear in the number of leaves and the reticulations of the network.

In \Cref{sec:unrooted}, we show how to define a 
complete invariant for the class of \vince{leaf-labelled rooted cactuses}, 
by generalising an approach used in \citep{liu2021tree}. 
\vince{We also explain how to use $Q$ to obtain} a 
complete invariant for \vince{(undirected) cactuses, either with
or without leaf-labels}. Finally, in \Cref{sec:discussion}, we discuss our findings and 
give suggestions for future research.

\section{Liu's polynomial revisited}\label{sec:prelim}

In this section, we define a polynomial invariant $P$
for rooted trees with certain vertex labellings which generalizes Liu's polynomial for 
rooted (leaf-labelled) trees.
We shall use this polynomial to define our new invariants for
cactuses in the following sections.

Let $\mathcal T$ denote the class of rooted trees.
For $T \in \mathcal T$ we let $\rho_T$ denote the root of $T$, $V(T)$ and
$L(T)$ be the vertex and leaf-sets of $T$ respectively, and 
$\bullet$ denote the single vertex tree in $\mathcal T$. 
The {\em stem} 
of a tree $T \in \mathcal T$ is 
the shortest directed path in $T$ that starts at the root of $T$ 
and does not end in a vertex of outdegree~$1$. We call a stem \emph{trivial} if it consists only of the root.
Note that if $T$ is the single vertex tree we regard it as a root and a leaf.
Given a set $\{T_1,\dots,T_k\}$, $k \ge 1$, of rooted trees we 
let $\wedge(T_1,\dots,T_k)$ denote the rooted tree obtained
by taking a single vertex $v$ and joining this vertex to the
root $\rho_i$ of each tree $T_i$ by an arc $(v,\rho_i)$.

Now, given an arbitrary rooted tree 
$T = \wedge(T_1,\dots,T_k)$ where $T_i \in \mathcal T$, define the 2-variable 
polynomial $B(T) \in \mathbb Z[x,y]$ by recursively 
applying the following rules \citep[Def. 2.1]{liu2021tree}: 
\begin{itemize}
\item[(1)] $B(\bullet)=x$, and
\item[(2)] $B(T) = y + \prod_{i=1}^k B(T_i)$. 
\end{itemize}

In \citep[Theorem 2.8]{liu2021tree} it is shown that $B$ is a
complete invariant for $\mathcal T$. Moreover 
in \citep[Corollary 3.5]{liu2021tree}, it is proven that a 
complete invariant $B_l$ for the class of leaf-labelled, rooted trees can be
defined as follows. For $T = \wedge(T_1,\dots,T_k)$ where $T_i \in \mathcal T$,
and $\bullet_j$ denoting the single vertex tree with label $x_j$,
define $B_l(T)$ by replacing rules (1) and (2) with \citep[Def. 3.4]{liu2021tree}:
\begin{itemize}
\item[(1')] $B_l(\bullet_j)=x_j$, and
\item[(2')] $B_l(T) = y + \prod_{i=1}^k B_l(T_i)$. 
\end{itemize}
Note that different leaf vertices in $T$ may have the same label.

We now generalize these polynomial invariants to 
rooted trees whose vertices are labelled. 
\rev{We call a maximal directed path of outdegree-1 vertices a \emph{string}.}

For any set of variables~$K$, we call a pair~$(T,\lambda)$ a \emph{vertex-labelled rooted tree} 
if~$T\in\cal{T}$ and~$\lambda:V(T)\to K$.
Let $K \subseteq \{x_1,x_2\dots\}\cup\{y,z\} \cup\rev{\{s\}}$,
be a fixed set of variables (possibly infinite).
We call the pair $(T,\lambda)$ {\em special} if:
\begin{itemize}
	\item[(i)] $\lambda(L(T)) \in \{x_1,x_2\dots\}$, 
	\item [(ii)] If $W= (w_1,\dots,w_p)$ is a string in $T$, then
	$\lambda(w_1) \in \{y,z\}$ and if $p >1$, then $\lambda(w_i)=y$ for all $2 \le i \le p$, and
	\item[(iii)] For every vertex $v \in V(T)$ that is not in $L(T)$ or in a string, $\lambda(v) \in \{y,z\}\cup S$.
\end{itemize}

Hence, basically, the vertices of a special pair can be labelled arbitrarily as 
long as the labelling of the leaves is disjoint from the labelling of the internal vertices and strings are labelled $(z,y,y,\dots)$ or $(y,y,\ldots)$. Moreover, we shall see below, given 
any $T \in \mathcal T$, there is some $\lambda$ so that $(T,\lambda)$ is a special pair.
Note also that if $T = \wedge(T_1,\dots,T_k)$ where $T_i \in \mathcal T$, then
$(T_i, \lambda_i=\lambda|_{V(T_i)})$ is a special pair 
for all $1 \le i \le k$.

We now define a polynomial invariant $P_K$ for the class of vertex-labelled trees.
Given a vertex-labelled tree 
$(T,\lambda)$ define $P_K(T,\lambda)$ recursively as follows:
\begin{itemize}
	\item If $T=(\{v\},\emptyset)$, set $P_K(T,\lambda) = \lambda(v)$. 
	\item If $T = \wedge (T_1,\dots,T_k)$, set $P_K(T,\lambda)= \lambda(\rho_T) + \prod_{i=1}^k P_K(T_i,\lambda_i)$.
\end{itemize}
To ease notation, in case the set $K$ is clear from the context, we denote $P_K(T,\lambda)$ by $P(T,\lambda)$.

An example of a special pair is given by setting $x=x_1$ and $K=\{x,y\}$, 
and, for any rooted tree $T \in \mathcal T$, defining $\lambda:V(T) \to K$
by $\lambda(v)=y$ if $v$ is an internal vertex and $\lambda(v)=x$ if
$v$ is a leaf. Then $P(T,\lambda)=B(T)$ is in fact Liu's polynomial for $T$.

For another example, set $K=\{x_1,x_2,\dots\} \cup\{y\}$ (note this is an infinite set).
Suppose $T \in \mathcal T$ is a leaf-labelled tree 
with $\phi: L(T) \to \{x_1,x_2,\dots,\leo{x_n}\}$ where $n \ge 1$. Then define $\lambda_L:V(T) \to K$, 
by setting $\lambda_L(v)=y$ if $v$ is an internal vertex and $\lambda_L(v)=\phi(v)$ if
$v$ is a leaf. Then $P(T,\lambda_L)=B_l(T)$ is Liu's polynomial for~$T$ with leaf-labels.

We say that special pairs $(T,\lambda)$ and $(T',\lambda')$ 
are \emph{isomorphic} if $T$ is isomorphic to $T'$ via a map $\psi:V(T) \to V(T')$
such that $\lambda(v)=\lambda'(\psi(v))$ for all $v \in V(T)$.
We now show that $P$ is a complete invariant for special pairs.

\begin{theorem}\label{variables}
	Suppose that $(T,\lambda)$ and $(T',\lambda')$ are special pairs.
$P(T,\lambda) = P(T',\lambda')$ if and only if $(T,\lambda)$ is isomorphic to $(T',\lambda')$.
\end{theorem}
\begin{proof}
It is straight-forward to check that $P$ is an invariant for 
special pairs using a similar proof to \citep[Prop 2.2]{liu2021tree}.
To show that $P$ is also a complete invariant, we use a similar 
proof to \citep[Cor. 3.5]{liu2021tree}. 

Suppose $P(T,\lambda) = P(T',\lambda')$. We want to show
that $(T,\lambda)$ is isomorphic to $(T',\lambda')$.
Assume $\lambda(V(T)) = \{s_1,\dots,s_k\}$   $\leo{\subseteq \{x_1,x_2\dots\}\cup\{y,z\} \cup S}$, where $k \ge 2$. 
Since $P(T,\lambda) = P(T',\lambda')$, it follows that $\lambda(V(T')) = \{s_1,\dots,s_k\}$
(i.e., $P(T,\lambda)$ and $P(T',\lambda')$ are both  polynomials in the same variables).

Consider the polynomial $P'$ in $\mathbb Z[x,y]$ 
obtained from the polynomial $P(T,\lambda)$ in $\mathbb Z[s_1,\dots,s_k]$
by replacing each variable $x_i$ by $x$ and each variable \leo{in $S\cup\{z\}$} by $y$.
Then $P'$ is Liu's polynomial for $T$, i.e. $B(T)$, and so $P'$ is 
irreducible in $\mathbb Z[x,y]$ \citep[Lem. 2.6]{liu2021tree}.
It follows that $P(T,\lambda)$ is irreducible in $\mathbb Z[s_1,\dots,s_k]$.

We now use strong induction on $n$, the number of leaf vertices in $T$.

If $n=1$, then $T=(w_1,\dots,w_l)$, where $l\ge 1$,  is a directed path of length $l-1$.  
If $l >1$, then $(w_1,w_2,\dots,w_{l-1})$ is a string in $T$.
If $l=1$, then $P(T,\lambda)=x_i$,
and \leo{if $l\geq 2$}, then $P(T,\lambda)=\leo{(l-1)}y+x_i$ or $P(T,\lambda)=z+\leo{(l-2)}y+x_i$, for some $i \ge 1$. In either 
case it is straight-forward to check that $P(T,\lambda)=P(T',\lambda')$ implies
$(T,\lambda)$ is isomorphic to $(T',\lambda')$.

Now suppose that~$T$ has~$N>1$ leaves and that the result holds for trees with $n$ leaves where $1 \le n \le N-1$. 
First note that if $T$ has a stem $(w_1,\dots,w_l,v)$ with length greater than 0 (with
outdegree of $v$ at least 2), then  $(w_1,\dots,w_l)$ is a string with $l \ge 1$. 
It follows that the degree 1 terms in  $P(T,\lambda)$ 
are either $ly+\lambda(v)$ or $z +(l-1)y +\lambda(v)$, where $\lambda(v) \in \{y,z\}\cup S$.
Hence, as  $P(T,\lambda)=P(T',\lambda')$,
it follows that $T'$ has a stem of the same length as $T$ and that $\lambda$ 
and $\lambda'$ restricted to the stems are the same.

It follows that we can assume that $T$ and $T'$ both have trivial stems consisting of a root with outdegree at least 2,
and hence that $\lambda(\rho_T), \lambda'(\rho_{T'}) \in \{y,z\}\cup S$.
So $T  = \wedge (T_1,\dots,T_i)$ where $i >1$, and $T'  = \wedge (T'_1,\dots,T_j')$ where $j>1$. 
But then
\begin{align*}
P(T,\lambda) &= \lambda(\rho_T) + \prod_{k=1}^i P(T_k,\lambda_k)\\
&= \lambda'(\rho_{T'}) + \prod_{l=1}^j P(T'_l,\lambda'_l)=P(T',\lambda').
\end{align*}
\noindent Since $i,j > 1$, the products are polynomials in which all terms have degree 2 or higher, and 
so $\lambda(\rho_T)= \lambda'(\rho_{T'})$ as these are both degree 1 polynomials. 
Moreover, $(T_k, \lambda_k)$, $1 \le k \le i$, and $(T'_l, \lambda'_l)$, $1 \le j \le l$, are all special pairs
and so their polynomials must be irreducible.
It follows that $i=j$ and after reordering of  indices, 
$P(T_k, \lambda_k) = P(T'_k, \lambda'_k)$ 
for all \leo{$1 \le k \le i$}. 
Since the trees $T_i$ and $T'_i$ have fewer than 
\leo{$N$} leaves, it follows  by the induction hypothesis that $T_i$ and $T'_i$ must be isomorphic. Therefore $(T,\lambda)$ is isomorphic to $(T',\lambda')$.
\end{proof}

\section{A rooted cactus invariant based on unfolding}\label{sec:unfolding}

In this section we modify the technique used in \citep{pons2022polynomial}
to obtain a complete polynomial 
invariant for a special subclass of rooted cactuses. To define this polynomial
we first require some additional notation.
Recall that a rooted cactus is a directed acyclic graph
with a single root $\rho$ whose
underlying undirected graph is a cactus, and 
such that there is a directed path from the root to any vertex in 
the graph. Note that from this definition, it follows 
that~$\rho$ is the only vertex with indegree~0 and 
that all vertices of~$V$ have indegree at most 2. Note that cycles of the underlying 
cactus may overlap in a vertex.
We call a vertex \emph{elementary} if it has indegree~1 and outdegree~1.
A vertex with outdegree~0 is called a \emph{leaf} 
and a vertex with indegree~2 is called a \emph{reticulation}
(so we allow a vertex to be both a leaf and a reticulation).
A \emph{reticulation cycle} for a 
reticulation $v$ in a cactus is a pair of internally vertex-disjoint paths ending in 
$v$ and with the same origin vertex $w$; we 
call~$w$ the \emph{top vertex} of the reticulation cycle.

We now define a polynomial invariant for the class of rooted cactuses, which we 
will show to be complete for rooted cactuses without elementary vertices.
First, we associate \vince{a vertex-labelled pair} $(U_N,\lambda_N)$ to a rooted cactus~$N$
as follows (see Figure~\ref{fig:cactus} for an example):
\begin{itemize}
\item[(A)]{For every vertex $v$ in~$N$ that is the top vertex for $c\geq 1$ reticulation cycles, 
resolve the vertex $v$ as follows. Let $w_1,w_1',\ldots, w_c,w_c'$ be children of~$v$ 
such that each pair~$(w_i,w_i')$ is in the same reticulation cycle. For $1\leq i\leq c$, 
add a new vertex~$v_i$ with an arc $(v,v_i)$, replace the arcs $(v,w_i),(v,w_i')$ 
by arcs $(v_i,w_i),(v_i,w_i')$ and give~$v_i$ label $s$, i.e., set $\lambda_N(v_i)=s$.
Call the resulting network $N'$.}\label{item:A}
\item[(B)] Unfold the network $N'$ to get $U_N$ as follows. While there exists a 
lowest reticulation~$v$, make a copy of the tree rooted at~$v$ and 
replace one incoming arc~$(w,v)$ of~$v$ by an arc~$(w,v')$ with~$v'$ the root of the copy of~$v$. 
If~$v$ and~$v'$ are leaves, label them~$r$. Otherwise, label them~$z$.
\item[(C)]{Give all leaf vertices in $U_N$, that are not already labelled, label $x$.
Give all remaining unlabelled vertices in $U_N$ label $y$.}\label{item:C}
\end{itemize}

\begin{figure}
\centerline{\includegraphics[width=\columnwidth]{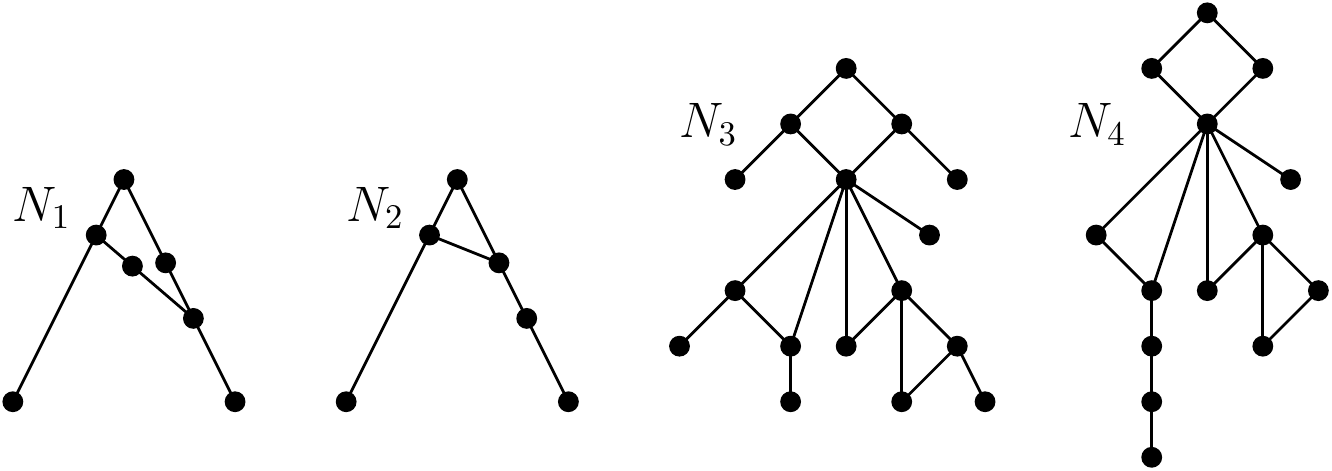}}
\caption{\rev{Example of four rooted cactuses. All arcs are directed downwards. The networks~$N_1$ and~$N_2$ have the same~$F$ polynomials~$F(N_1) = F(N_2) = y + s + \left[y+x(y+z+x)\right]\left[y+z+x\right]$, but different~$Q$ polynomials, since $Q(N_1) = y+s+[y+x(y+q)](z+x)(y+q)$ and $Q(N_2) = y+s+ (y+xq)(z+y+x)q$ (for example, $Q(N_1)$ has an~$xyzq$ term while~$Q(N_2)$ does not).}
\label{fig:cactus}}
\end{figure}

Observe that each vertex of~$U_N$ that corresponds to a reticulation of~$N$ is either an 
internal vertex labelled~$z$ or a leaf labelled~$r$. We now show that 
$(U_N,\lambda_N)$ is a special pair if~$N$ has no elementary vertices.

\begin{lemma}\label{lem:Fspecialpair}
Let~$N$ be a rooted cactus without elementary vertices. Then 
$(U_N,\lambda_N)$ is a special pair with $K \subseteq \{x_1=x,x_2=r\} \cup\{y,z\}\cup\{s\}$.
\end{lemma}
\begin{proof}
Each string consists of exactly one vertex, which is labelled~$z$. In addition, each leaf is labelled~$x$ or~$r$ and each internal vertex is labelled~$y,z$ or~$s$.
\end{proof}

Note that \Cref{lem:Fspecialpair} would not be true if we allowed~$N$ to contain elementary vertices, as then~$(U_N,\lambda_N)$ could contain a string with a sequence of labels~$(y,z,\ldots)$. \rev{See \Cref{fig:cactus} for two non-isomorphic networks that have the same~$F$ polynomials.}
We now define operations that undo the unfolding described above. The \emph{lowest common ancestor (LCA)} of two vertices~$u$ and~$v$ of a rooted cactus is the unique lowest vertex that has a path to both~$u$ and~$v$.
Given $(U_N,\lambda_N)$, for a rooted cactus~$N$ without elementary vertices, we construct a 
digraph $M(U_N,\lambda_N)$ from the rooted tree $U_N$ as follows.

\begin{itemize}
\item[(a)] While there exists a pair of vertices $a,b$ in $U_N$ 
that are either both internal vertices labelled~$z$ or both leaves 
labelled~$r$, choose such a pair with lowest
 LCA
$v_t$.
Fold-up the two rooted cactuses below $a$ and $b$ by adding an arc from 
the parent of~$b$ to~$a$ and then deleting~$b$ and all its descendants. Remove the label ($r$ or~$z$) from~$a$.
\item [(b)] Contract all arcs of the form $(v,u)$ where $u$ has label~$s$.
\end{itemize}
\rev{Observe that this digraph~$M(U_N,\lambda_N)$ is unlabelled as it does not come with a labelling map.}

\begin{lemma} \label{unfolding-reverse}
If $N$ is a rooted cactus without elementary vertices, then 
$N$ is isomorphic to $M(U_N,\lambda_N)$.
\end{lemma}
\begin{proof}

We will show that (a) reverses (B). Since it is clear that (b) reverses (A),
and since the vertices of $N$ and $M(U_N,\lambda_N)$ are unlabelled,
it will follow that $M(U_N,\lambda_N)$ is isomorphic to~$N$.

We now show that (a) reverses (B). Each iteration of (B), acting 
on a reticulation~$v$ of a reticulation cycle with top vertex~$v_t$, 
creates a pair~$a,b$ of vertices that are either both internal vertices 
labelled~$z$ or both leaves labelled~$r$. In either case, their lowest common ancestor is~$v_t$. 
Call such a triple~$(a,b,v_t)$ a \emph{good triple}. Future 
iterations may make copies of such good triples, which will also all be good. 
We will show that (a) only selects good triples.

To do so, we first observe that, for any pair of 
good triplets $(a,b,v_t),(a',b',v_t')$ (with $\{a,b\}\neq\{a',b'\}$), the 
paths from~$v_t$ to~$a$ and~$b$ are vertex disjoint from the 
paths from~$v_t'$ to~$a'$ and~$b'$. This follows from the observation that, after step (A), 
all reticulation cycles are vertex disjoint.
We now show that (a) only selects good triples. Suppose (a) selects a triple $(a,b,v_t)$ that is not good. 
Since~$a$ is labelled~$z$ or~$r$, it is also in a good triple $(a,c,w_t)$. 
Note that~$w_t$ is not strictly below~$v_t$ because then (a) would not select $(a,b,v_t)$. 
So either $w_t=v_t$ or $w_t$ is above~$v_t$.

Now note that also~$b$ is in a good triple $(b,d,u_t)$. Moreover,~$u_t$ is below~$v_t$, 
and $u_t\neq v_t$, because otherwise the paths from~$u_t$ to~$b$ and~$d$ would not be 
vertex disjoint from the paths from~$w_t$ to~$a$ and~$c$.
This gives a contradiction because in this case (a) would not select $(a,b,v_t)$.

We conclude that (a) only selects good triples. Hence, each iteration of (a) 
undoes an iteration of (B) on one copy of the created good triple. It follows that (a) reverses (B).
\end{proof}

Now define a polynomial $F$ 
for a rooted cactus $N$ 
by setting

$$
F(N)=P(U_N,\lambda_N).
$$ 

\noindent Note that if $N \in \mathcal T$ is a rooted tree, 
then $F(N)=B(T)$, i.e., Liu's polynomial for $T$.
Moreover, if $N$ is \emph{binary}, i.e., if all vertices have outdegree at most~$2$ 
and all reticulations have outdegree~1, then 
$F$ is a polynomial in the four variables $x,y,z$ and~$s$, and 
otherwise in the five variables $x,y,z,r$ and~$s$.

We now show that~$F(N)$ is a complete invariant for 
the class of rooted cactuses without elementary vertices.

\begin{theorem}\label{thm:cactus-invariant}
	If $N$ and $N'$ are rooted cactuses without elementary vertices, then
	$F(N)=F(N')$ if and only if $N$ is isomorphic to $N'$.
\end{theorem}

\begin{proof}
It is straight-forward to check that if $N$ is isomorphic to $N'$, 
then $(U_N,\lambda_N)$ is isomorphic to $(U_{N'},\lambda_{N'})$. So
by Theorem~\ref{variables} $F(N)=F(N')$.

Conversely, suppose $F(N)=F(N')$.
By Lemma~\ref{lem:Fspecialpair},~$(U_N,\lambda_N)$ and~$(U_{N'},\lambda_{N'})$ are both special pairs.
Therefore by Theorem~\ref{variables},
$(U_N,\lambda_N)$ is isomorphic to $(U_{N'},\lambda_{N'})$.
So $M(U_N,\lambda_N)$ is isomorphic to $M(U_{N'},\lambda_{N'})$.
By Lemma~\ref{unfolding-reverse} it now follows that $N$ is isomorphic to $N'$.
\end{proof}

\begin{remark}
Observe that the polynomial~$F(N)$ can be obtained directly from~$N$ by defining 
a polynomial~$F(N,v)$ for each vertex~$v$ of~$N=(V,A)$ as follows, and letting~$F(N)=F(N,\rho)$.
\begin{itemize}
	\item If~$v$ is a leaf, then
	\[
	F(N,v) = \begin{cases} r &\quad \mbox{if~$v$ is a reticulation,}\\
	x &\quad \mbox{otherwise.}
	\end{cases}
	\]
	\item Otherwise, $v$ is the top vertex of~$c\geq 0$ reticulation cycles, with $w_1,w_1',\ldots, w_c,w_c'$ children of~$v$ such that each pair~$(w_i,w_i')$ is in the same reticulation cycle, and~$w_{c+1},\ldots ,w_{d-c}$ the other children of~$v$ (with~$d$ the outdegree of~$v$). Let~$G(N,v) = \prod_{i=1}^{c} \left[ s + F(N,w_i)F(N,w_i') \right] \prod_{i=c+1}^{d-c} F(N,w_i)$. Then,
	\[
	F(N,v)= \begin{cases} z + G(N,v) &\quad\mbox{ if~$v$ is a reticulation}\\
	 y + G(N,v)&\quad\mbox{ otherwise.}
	 \end{cases}
	\]
\end{itemize}
\end{remark}

\section{A rooted cactus invariant based on expanding}\label{sec:expanding}

In the last section, we defined a polynomial $F$ that is a complete invariant for the special class of
rooted cactuses without elementary vertices. Although this polynomial has the advantage
of being in at most 5 variables, it is not a complete invariant for all rooted cactuses, 
and~$F$ could have degree that is exponentially dependent on the number of reticulations
(observe that copies of networks rooted at each reticulation are 
created when obtaining the special pair).
In this section, we introduce a polynomial invariant $Q$ for the class of rooted cactuses, 
which has the advantages of not needing to exclude elementary vertices
and having a linearly bounded degree.

To define $Q$ and show that it is a complete 
invariant for cactuses, we use a similar approach to the one
used for $F$.
Suppose that $N=(V,E)$ is a rooted cactus. We begin by associating a 
vertex-labelled tree $(T_N,\mu_N)$ to $N$ as follows (see Figure~\ref{fig:poly} for an example):
\begin{itemize}
\item[(A)] As in page \pageref{item:A}.
\item[(B')] Expand the network~$N'$ to get $T_N$ as follows. While there exists a reticulation, 
take a lowest reticulation cycle in $N'$, with top vertex $v_t$ and with reticulation $v$ having parents $u,u'$.
Remove arcs $(u,v)$, $(u',v)$. Add vertices $p,p'$ and arcs $(u,p)$, $(u',p')$. Add
arc $(v_t,v)$. Give both of the vertices $p,p'$ (which are leaves) label~$q$.
Give $v$ the label $z$ if it is an internal vertex and $r$ if it is a leaf.
\item[(C)] As in page \pageref{item:C}.
\end{itemize}

\begin{figure*}
\centerline{\includegraphics[scale = 0.7]{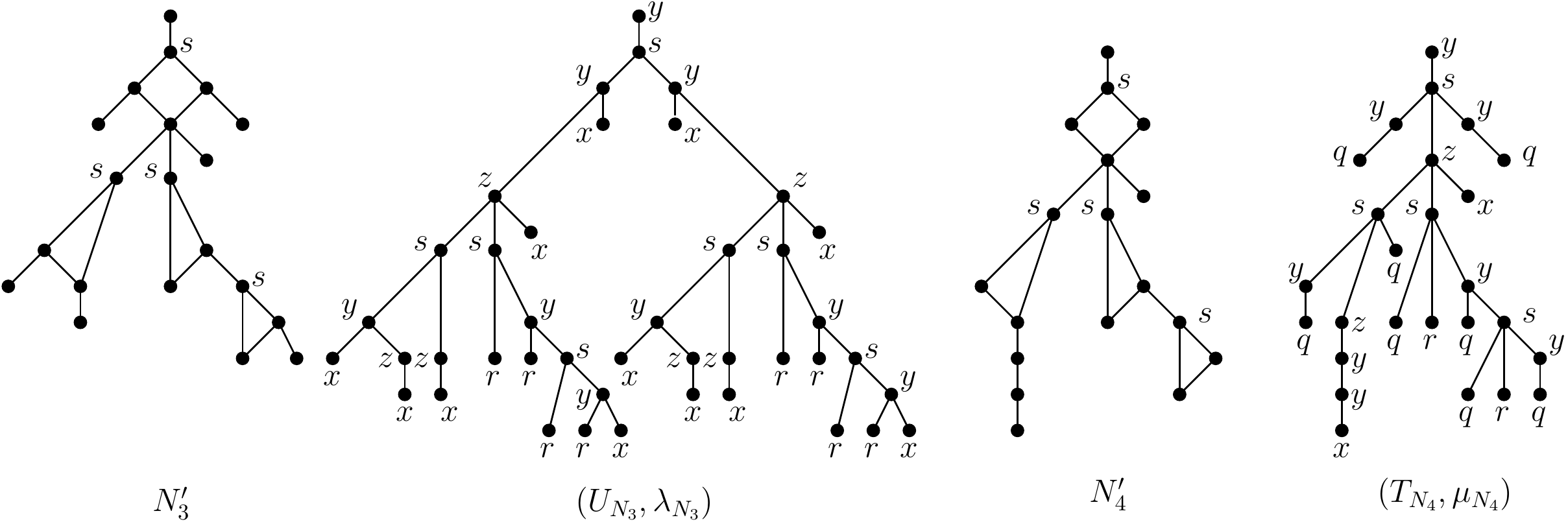}}
\caption{\rev{Example of how the special pairs $(U_{N_3},\lambda_{N_3})$ and~$(T_{N_4},\mu_{N_4})$ are obtained from 
 rooted cactuses~$N_3$ and~$N_4$ of \Cref{fig:cactus}, respectively. 
All arcs are directed downwards.}\label{fig:poly}}
\end{figure*}

The following lemma is straight-forward to prove.

\begin{lemma}\label{lem:Qspecialpair}
$(T_N,\mu_N)$ is a special pair with $K \subseteq \{x_1=x,x_2=r,x_3=q\} \cup\{y,z\}\cup\{s\}$. 
\end{lemma}

Now, given $(T_N,\mu_N)$, we construct a 
digraph $M(T_N,\mu_N)$ from the rooted tree $T_N$ as follows:
\begin{itemize}
\item[(a')] While there exists a pair of leaves $p,p'$ labelled~$q$ in $T_N$, 
choose such a pair with lowest LCA $v_t$. Let~$u,u'$ be the (necessarily unique) 
parents of~$p$ and~$p'$, respectively. Let~$v$ be the (necessarily unique) 
child of~$v_t$ labelled~$z$ or~$r$. Remove leaves~$p,p'$ and arc~$(v_t,v)$. Add arcs~$(u,v),(u',v)$.
\item [(b)] Contract all arcs of the form $(v,u)$ where $u$ 
has label~$s$.
\end{itemize}

\begin{lemma} \label{reverse}
If $N$ is a rooted cactus, then $N$ is isomorphic to $M(T_N,\mu_N)$.
\end{lemma}
\begin{proof}
The proof follows from observing that (a') reverses (B), that (b) reverses (A) and that the vertices of both~$N$ and $M(T_N,\mu_N)$ are unlabelled.
\end{proof}

Now define a polynomial $Q$ 
for a rooted cactus $N$ by setting
\begin{equation}\label{Qequation}
Q(N)=P(T_N,\mu_N).
\end{equation}
Note that if $N \in \mathcal T$ is a rooted tree, then 
$Q(N)=F(N)=B(T)$. Moreover, if $N$ is binary, then $Q$ is a polynomial 
in the five variables $x,y,z,q,s$ and otherwise in the six variables $x,y,z,q,r$ and~$s$.

We now show that~$Q(N)$ is a complete invariant for rooted cactuses.

\begin{theorem}\label{thm:cactus-invariant2}
	If $N$ and $N'$ are rooted cactuses, then
	$Q(N)=Q(N')$ if and only if $N$ is isomorphic to $N'$.
\end{theorem}
\begin{proof}

This is analogous to the proof of \Cref{thm:cactus-invariant}, by replacing~\Cref{unfolding-reverse,lem:Fspecialpair} with \Cref{reverse,lem:Qspecialpair}.
\end{proof}

Note that, in contrast to the polynomial $F$, $Q$ less directly
reflects the structure of a cactus. Indeed, 
with the construction of $Q$, we add two additional leaves for every reticulation.
On the other hand, with the construction used for~$F$, every 
leaf of the unfolded tree corresponds to a leaf of the original network.
However, as we shall now show, unlike $F$,
the degree of $Q$  can linearly 
bounded in terms of the number of leaves and reticulations.

\begin{theorem}\label{lem:QDegree}
Let~$N$ be a rooted cactus with~$n$ leaves and~$k$ reticulations. 
Then~$Q(N)$ is a polynomial 
of degree~$n+2k$.
\end{theorem}
\begin{proof}
We have defined~$Q(N)$ as 
the polynomial $P$ of the special pair~$(T_N,\mu_N)$. 
By a similar argument to the one presented in [\citep{liu2021tree}, Lemmas 2.3 and 2.4], 
each term of the polynomial $Q(N)$ corresponds to a rooted subtree~$S_N$ of~$T_N$, 
where~$S_N$ and~$T_N$ share the same root, and any leaf of~$T_N$ is a leaf of~$S_N$ 
or a descendant of a leaf of~$S_N$
(for example, the term~$cx^ar^bq^dy^ez^fs^g$ in~$Q(N)$ corresponds to a primary subtree of~$T_N$ 
that contains $a+b+d+e+f+g$ leaves: $a$ labelled $x$, $b$ \rev{labelled} $r$, $d$ \rev{labelled} $q$, $e$ \rev{labelled} $y$, $f$ \rev{labelled} $z$, and $g$ \rev{labelled} $s$).
Such rooted subtrees are called \emph{primary subtrees}.
The coefficient~$c$ of the term counts exactly how many primary subtrees exist with matching leaves.
The degree of the polynomial $Q(N)$ is then the greatest number of 
leaves across all primary subtrees of~$T_N$. 
This corresponds to, amongst possibly other primary subtrees, the primary subtree 
that is isomorphic to~$T_N$.
To count the number of leaves in~$T_N$, observe that we add two 
leaves labelled~$q$ every time we expand a reticulation cycle.
It follows that the number of leaves in~$T_N$, and therefore the degree of the polynomial~$Q(N)$, is~$n+2k$. 
\end{proof}

\begin{remark}
As with the polynomial $F$, we
observe that the polynomial~$Q(N)$ can be obtained directly from~$N$ by 
defining a polynomial~$Q(N,v)$ for each vertex~$v$ of~$N=(V,A)$ as follows, and letting~$Q(N)=Q(N,\rho)$.
\begin{itemize}
	\item If~$v$ is a leaf then
	\[
	Q(N,v) = \begin{cases} r &\quad \mbox{if~$v$ is a reticulation,}\\
	x &\quad \mbox{otherwise}
	\end{cases}
	\]
	\item Otherwise, $v$ is the top vertex of~$c\geq 0$ reticulation cycles, with $w_1,w_1',\ldots, w_c,w_c'$ children of~$v$ such that each pair~$(w_i,w_i')$ is in the same reticulation cycle with reticulation~$v_i$, and~$w_{c+1},\ldots ,w_{d-c}$ the other children of~$v$ (with~$d$ the outdegree of~$v$). Let~$R(N,v) = \prod_{i=1}^{c} [s + Q'(N,w_i)Q'(N,w_i')Q(N,v_i)] \prod_{i=c+1}^{d-c}Q'(N,w_i)$ where
	\[
	Q'(N,v) =
	\begin{cases}
	q &\quad\mbox{ if~$v$ is a reticulation,}\\
	Q(N,v) &\quad\mbox{ otherwise.}\\
	\end{cases}
	\]
	Then we have
	\[
	Q(N,v)=
	\begin{cases}
	    z + R(N,v)
	        &\quad\mbox{ if~$v$ is a reticulation,}\\
	    y + R(N,v)
	        &\quad\mbox{ otherwise,}
	\end{cases}
	\]
\end{itemize}
\end{remark}
\rev{See \Cref{fig:cactus} for an example of a~$Q$ polynomial for a given network.}

\section{Cactuses and leaf-labellings}\label{sec:unrooted}

In this section, we show that the invariant $Q$ can be used to give 
an invariant for cactuses (which are undirected), and that it is 
relatively simple to extend its definition so as
to give an invariant for partially leaf-labelled rooted cactuses and cactuses
(note that for cactuses a leaf is a vertex with degree 1).

First, recall, that a cactus is an (undirected) graph in which any two cycles are 
edge-disjoint. For a cactus~$G$, let $R(G)$ be the set of
all rooted cactuses that can be obtained from~$G$ by picking some vertex to be the root 
and choosing some orientation for each edge in $G$. Note that~$R(G)$ may contain isomorphic rooted cactuses.
The following observation follows a similar idea to the one used in \citep[Lemma 3.1]{liu2021tree}. 

\begin{lemma}\label{lem:direct}
Suppose that $G,G'$ are undirected cactuses.
Then $G$ and $G'$ are isomorphic if and
only if $R(G)\sim R(G')$ (i.e., there is a bijection $f:R(G)\to R(G')$ 
such that $f(D)$ is isomorphic to~$D$ for all~$D\in R(G)$).
\end{lemma}
\begin{proof}
If~$G$ and~$G'$ are isomorphic, then clearly $R(G)\sim R(G')$. 
Conversely, if $R(G)\sim R(G')$ then pick any~$D\in R(G)$. 
Then there is a~$D'\in R(G')$ that is isomorphic to~$D$. 
Since~$G$ and~$G'$ are the underlying undirected graphs of~$D$ and~$D'$, respectively, 
it follows that~$G$ and~$G'$ are isomorphic.
\end{proof}

Now suppose that $C$ is any complete polynomial invariant for rooted cactuses 
with $C(N) \in \mathbb Z[t_1,\dots,t_m]$ for some variables $t_1,\dots,t_m$, 
and $C(N)$ is irreducible in $\mathbb Z[t_1,\dots,t_m]$, for any rooted cactus $N$.
Given a cactus $G$, we define the 
polynomial $C_u$ by
$$
C_u(G) = \prod_{D \in R(G)} C(D). 
$$

Note that this polynomial can have large degree in general, 
depending on the number of cycles and the size of each cycle 
(larger sizes lead to more rootings).
The following result is related to the observation 
following \citep[Theorem 3.2]{liu2021tree}.

\begin{theorem}
The polynomial $C_u$ is a complete invariant for cactuses.
\end{theorem}
\begin{proof}
By assumption, for each $D \in R(G)$, $C(D)$ is an irreducible 
polynomial in $\mathbb Z[t_1,\dots,t_m]$. 
The result now follows from Lemma~\ref{lem:direct} and the fact that
$\mathbb Z[t_1,\dots,t_m]$ is a unique factorization domain.
\end{proof}

\begin{corollary}
For the polynomial $Q$ defined in 
(\ref{Qequation}), $Q_u$ is a complete invariant for cactuses.
\end{corollary}

Note that $Q_u$ restricted to the class of rooted trees is not the
same as Liu's polynomial since our rooting procedure differs 
the one in \citep{liu2021tree}.

We now consider \emph{leaf-labelled} rooted cactuses \vince{(see e.g. \citep{huber2021space})}. 
As mentioned in the introduction, 
in phylogenetics it is common to consider (di)graphs 
in which the leaves are labeled by some set of species \vince{(see e.g. \cite[Section 10.3]{steel2016phylogeny}).}
Using our results on special pairs it is straight-forward to 
obtain a complete invariant for partially 
leaf-labelled rooted cactuses (i.e. ones in which some subset of the 
leaves is labelled), in a similar
way used to deal with leaf-labelled trees in \citep{liu2021tree}.
Indeed, given a rooted cactus $N$ and a map $\phi : L' \to \{x_1,x_2,\dots, x_p\}$ 
from a subset~$L'$ of the leaves of $N$ to some 
collection of species $\{x_1,x_2,\dots x_p\}$,
we can define a special pair $(T_N,\mu_N)$ as  
in the definition of $Q$ above, except that in step (C), we give all
remaining leaf vertices $v$ the label $x_i$ (instead of $x$) if
$\phi(v)=x_i$ and label~$x$ (as before) if~$v\notin L'$. 
Then using the same argument as in Theorem~\ref{thm:cactus-invariant2}, it
can be seen that this leads to a complete invariant for 
the class of partially leaf-labelled rooted cactuses.
Moreover, using a partially leaf-labelled version of 
Lemma~\ref{lem:direct}
(in which we only select roots at vertices that are unlabelled), it is possible to also 
define a polynomial invariant for the collection of partially leaf-labelled cactuses
\vince{(see e.g. \cite{hayamizu2020recognizing})},
and to show that it is a complete invariant for this class.

\section{Discussion}\label{sec:discussion}

We have introduced a new complete polynomial invariant $Q$ for the class of rooted cactuses 
which has at most 6 variables, and also shown how to use $Q$ to obtain a complete invariant 
for cactuses. In addition, we have shown that the degree of the 
$Q$ polynomial for a rooted cactus with~$n$ leaves and~$k$ reticulations is~$n+2k$ [\Cref{lem:QDegree}].
We have also introduced a polynomial invariant~$F$ for rooted cactuses which 
has only $5$ variables regardless of the number of reticulations in the network,
and more naturally respects the structure of the network. However,
$F$ is only a complete invariant for rooted cactuses without elementary vertices
and the degree of $F$ for a rooted cactus may be exponential in 
the number of reticulations in the rooted cactus.


It could be of interest to look for other polynomial invariants of rooted cactuses/cactuses
which shed different light on their structure. 
For example, in \citep{gordon1989greedoid} it is asked if 
the greedoid polynomial invariant introduced for rooted trees gives 
a complete invariant for larger classes of directed rooted graphs,
and so it is natural to ask if it gives a complete invariant for rooted cactuses.
However, in recent work \citep{yow2019discovering}
it is shown that, even for a simple example of 
a rooted cactus \citep[Fig. 3]{yow2019discovering}, the polynomial does not  
reflect the structure of the directed graph,
so this seems unlikely to be the case. Another 
possibility would be to consider the TVT polynomial in \citep{tedford2009tutte}
or the ``B-polynomial" for digraphs 
introduced in \citep{awan2020tutte}; the TVT polynomial 
gives a complete invariant for rooted trees \vince{\citep[p.569]{tedford2009tutte} but it is not
known whether this is the case for the B-polynomial \citep[Question 10.6]{awan2020tutte}.}

Another natural question is to look for complete invariants for broader
classes of (di)graphs that are of interest in phylogenetics.
For example, we could consider the classes $\mathcal C$ and $\mathcal C'$
that consist of directed acyclic graphs that 
are uniquely determined by the graphs 
obtained by applying steps (A),(B), and (C) or (A), (B'), and (C),
used to obtain a special pair in the definition of $F$ and $Q$, respectively.
It would be interesting to know if we could obtain
complete invariants for $\mathcal C$ and $\mathcal C'$
in a similar way to $F$ and $Q$ and, if so, which well-known classes  
of (non leaf-labelled) phylogenetic networks are contained in these classes
(note that for generalizing the approach used to obtain $F$, results on so-called
stable networks in~\citep{huber2016folding} might be relevant).
More specifically, it would be interesting to know for which $k \ge 0$ 
is the collection of level-$k$ networks
(network in which every biconnected component contains 
at most $k$ reticulations) contained in either $\mathcal C$ or $\mathcal C'$?
Note that (non-leaf-labelled) level-0 and level-$1$ networks are rooted trees and 
cactuses, respectively, so $k=2$ is the first case of interest.
\rev{Extra care must be taken in this $k=2$ case, as reticulation cycles may have more than one top vertex. In particular, this means unfolding and expanding will have to be redefined, perhaps depending on the comparability of top vertices of reticulation cycles.}

As we have seen, different constructions of polynomials give rise to advantages and disadvantages 
with respect to the number of variables, the polynomial size, and the 
graph classes for which it is complete. Ideally, it would be useful to 
find complete invariants that can be computed in polynomial time for
special classes of phylogenetic networks\footnote{This may be too much to ask in general
in light of the fact DAG isomorphism checking is at least as hard as the graph 
isomorphism problem}. However, even if it is not possible 
to find such invariants, it could still 
be interesting to look for polynomials like the Tutte polynomial which
may not provide complete invariants but may be easier to compute and can still 
provide useful information about the underlying structure of the network.



\bibliographystyle{elsarticle-harv}
\bibliography{alpha}




\end{document}